\documentclass[final,1p,times]{elsarticle}

\usepackage{amssymb}
\usepackage{amsmath}
\usepackage[utf8]{inputenc}
\usepackage{amsmath}
\usepackage{amsfonts}
\usepackage{color,soul}
\usepackage{placeins}
\usepackage{makecell}
\usepackage{amssymb}
\usepackage{graphicx}
\usepackage{hyperref}
\begin{document}
\newtheorem{thm}{Theorem}[section]
\newtheorem{lem}{Lemma}[section]
\newtheorem{rem}{Remark}[section]
\newtheorem{prop}{Proposition}[section]
\newtheorem{cor}{Corollary}[section]
\newtheorem{example}{Example}[section]
\newdefinition{defn}{Definition}[section]
\newproof{proof}{Proof}
\renewcommand{\theequation}{\thesection.\arabic{equation}}

\begin{frontmatter}



\title{Probabilistic Modelled Optimal Frame for Erasures
under Spectral and Operator Norm}

\author {Shankhadeep Mondal}
\ead{shankhadeep16@iisertvm.ac.in}

\address{School of Mathematics, Indian Institute of Science Education and Research Thiruvananthapuram, Maruthamala P.O, Vithura, Thiruvananthapuram-695551.}

\begin{abstract}
    Error occurs in data transmission process when some data are missing at the time of reconstruction. Finding the best dual frame or a dual pair that minimizes the reconstruction error when erasure occurs,is a deep-rooted problem in frame theory. The main motivation behind this paper is to characterize the optimal dual under the spectral and operator norms. Here we give several equivalent conditions for which the canonical dual is an optimal dual frame for a given frame. We also go on to characterize the set of dual pairs which attains the optimal value.
\end{abstract}


\end{frontmatter}

\section{Introduction}

Frames are used in data transmission due to their redundancy features. Errors occur in data  transmission if data is missing at some positions. So redundancy can help to reduce errors of the reconstructed signal. Moreover, it is usually more flexible to construct frames rather than an orthogonal  basis or Riesz basis. Recently a lot of work has been done towards erasures( \cite{leng},\cite{casa},\cite{goyal},\\\cite{sal},\cite{bodmann},\cite{casa2},\cite{leng3},\cite{jer},\cite{jin}, \cite{ara}, ).\\
Optimal dual problem deals with minimizing the maximum error for erasures. The erasure problem was first introduced by Paulson and Holmes( \cite{holmes} ). The concept of finding the optimal dual frame has two approaches : first by Lopez and D. Han ( \cite{jer} ) and another one by Pahlivan, D.Han and Mohapatra   ( \cite{sal} ).\\
In the transmission process, data erasure arises from buffer overflows at the routers. Some times, error occurs from the bad global conditions of the network such as conjunction, capacity of transmission channel etc. In these cases, erasures of different elements usually generate some probabilistic irregularities. Probability of a bad channel failure is usually larger than the probability of a good channel failure.\\
In the paper \;\cite{leng}, authors setup the probability model to characterize the optimal dual frames for a given frame. Several researchers have also worked on probability modelled erasures for optimal K-frame ( \cite{miao}). In this paper, we have worked on probability modelled optimal dual pair under spectral and operator norm for one and two erasures. The outline of this paper is as follows. In section 3, we give a setup for probabilistic spectrally optimal dual pair. We also propose certain conditions for which the canonical dual will be one-erasure probabilistic spectrally optimal dual for a given frame $F$. We have also established a relation between one and two erasure probabilistic spectrally optimal dual frames. More generally, we have tried to characterize the probabilistic modelled spectrally optimal dual pair. In section 4, we give certain conditions for which a canonical dual will be a probabilistic optimal dual for one erasure and try to characterize all one-erasure probabilistic optimal duals for a given frame $F$. At last, we have shown a bridge between one-erasure probabilistic spectrally optimal dual and one-erasure probabilistic optimal dual.

\section{Preliminaries on Erasures for probability model }

Let $H$ denote an $n-$dimensional(real or complex) Hilbert space $H.$  A finite sequence of elements $F= \{f_k\}_{k=1}^N$ in $H$ is called a \textit{Frame} for $H$ if there exist constants $A,B >0$ such that
 $$\displaystyle{A \left\| f \right\|^2\leq \sum_{i=1}^N\big|\langle f,f_i\rangle\big|^2 \leq B\left\| f \right\|^2},   \forall f\in H.$$

 The numbers $A$ and $B$ are called frame bounds. They are not unique. The  \textit{optimal lower frame bound} is the supremum over all lower frame bounds and the \textit{optimal upper frame bound} is the infimum over all upper frame bounds.    The frame is \textit{normalized} if $\left\|f_i\right\|=1 $  for every $ i$.  If $A=B,$ i.e.  $\displaystyle{\sum_{i=1}^N\big| \langle f,f_i \rangle \big|^2 = A\left\| f \right\|^2}$ for all $f\in H,$  then $\{f_i\}_{i=1}^N$ is called a Tight Frame.
     If $A=B=1,$ then $\{f_i\}_{i=1}^N$  is called a \textit{Perseval frame}.
Every finite sequence  $\{f_i\}_{i=1}^N$  in $H$ is a frame for the Hilbert space W:= span $\{f_i\}_{i=1}^N$.
 Let $H$ be a Hilbert space equipped with a frame $F = \{f_i\}_{i=1}^N.$ Then the linear mapping $\Theta_F: H \to  {\mathbb{C}^N} $ defined by $$\Theta_F(f)= \{\langle f,f_i \rangle\}_{i=1}^N $$  called the  \textit{analysis operator}.\\
      The adjoint operator   $ \Theta_{F}^*: \mathbb{C}^N \to H $  defined by
        $$\Theta_{F}^*\left(\{c_i\}_{i=1}^N\right) =  \sum_{i=1}^N {c_i f_i} $$ is called \textit{synthesis operator or preframe operator}.
        The \textit{Frame operator} $S_F : H\to H $ is defined by
        $$S_{F}f= \Theta_{F}^{*} \Theta_{F} f = \sum_{i=1}^N \langle f,f_i \rangle f_i $$
        which is a positive, self adjoint, invertible operator on $H$ and which leads to the reconstruction formula :
        $$f= \sum_{i=1}^N \langle f,S_{F}^{-1}f_i\rangle f_i\ .$$

  A frame $G=\{g_i\}_{i=1}^N$ in $H$ is called a dual frame of $F= \{f_i\}_{i=1}^N$ if every element $f \in  H$ can be written as $$ f = \sum_{i=1}^N\langle f,f_i \rangle g_i  = \sum_{i=1}^N \langle f,g_i \rangle f_i  \;\;   \forall f\in H.$$

\noindent
It is known that $\{S_{F}^{-1}f_i\}_{i=1}^{N}$ is  a frame and is called the canonical or standard dual frame.   There exist infinitely many dual frames (\cite{ole}) $G$ of $F$ in $H$ and
every dual frame $G=\{g_i\}_{i=1}^N$ of F is of the form $G=\{S_{F}^{-1}f_i +u_i\}_{i=1}^N,$ where the sequence $\{u_i\}_{i=1}^N$ satisfies $$\sum_{i=1}^N \langle f,f_i \rangle u_i = \sum_{i=1}^N \langle f,u_i \rangle f_i = 0 \;\;   \forall f\in H.$$

\noindent
 For any dual frame $G$ of $F,$  we have
 $\displaystyle{\sum_{i=1}^N\langle g_i,f_i \rangle} = tr(T_F T_{G}^*)=tr (T_{G}^* T_F)= tr(I)=n.$  In particular if F is a Parseval frame, then $\displaystyle{\sum_{i=1}^N \left\|f_i\right\|^2 =n }$.\\
  
  Let $F = \{f_i\}_{i=1}^N $ be a frame for $H$ and $G = \{g_i\}_{i=1}^N$ be a dual of $F.$ Then $(F,G)$ is called an $(n,N)$ dual pair for $H.$\\

In this paper, we have tried to investigate the case where the erasures of the transmitted vector $T_{F}f$ satisfy some probability distribution regularity. So it is considered that during data transmission, the probability of element loss in a bad transmission channel is larger than the one in a good channel.\\
Let $\{p_i \}_{i=1}^N$ be the probability sequence, where $p_i$ be the probability of the i'th erasure for $i=1,2,...,N.$ Then $\{p_i \}_{i=1}^N$ must satisfy
 \begin{equation} \label{eqn2.1}
     \displaystyle{\sum_{i=1}^N p_i = 1}, \;\;\; 0 \leq p_i \leq 1 ,\;\; i=1,2,...,N. 
 \end{equation}
The weight number $q_i$ is defined as follows :\\
 \begin{equation} \label{eqn2point1}
 q_i = \displaystyle{{\frac{\sum_{j=1}^N p_j}{\sum_{j=1}^N p_j - p_i}\cdot \frac{N-1}{n}}} ; \;\;\; for\;\; i= 1,2,...,N\;\;\;\;
\end{equation}

\begin{prop}
Let $H$ be an Hilbert space of dimension n and $N \geq n.$  Let $\{q_i\}_{i=1}^N$ be a weight number sequence is given by  (\ref{eqn2point1}) for a frame $F.$ Then $\{q_i\}_{i=1}^N$ satisfy the following properties :
\begin{enumerate}
    \item [{\em (i)}] $\displaystyle{q_i \geq 1 ,\;\; \forall 1 \leq i \leq N }$
    \item [{\em (ii)}] $\displaystyle{\sum_{j=1}^N \frac{1}{q_i}} = n $
    \item [{\em (iii)}] If the  number $p_i$ increase then the number $q_i$ also increase.
\end{enumerate}
\end{prop}  \hfill{$\square$}

\vskip 1em
\noindent
 During data transmission if error occurs in \textit{`m' positions} then the error operator is defined by
   $$E_{\Lambda}f:= \Theta_{G}^*D_{p}\Theta_{F} f=\sum_{i\in\Lambda} q_i \langle f,f_i \rangle g_i, $$
   where $\Lambda $ is the set of indices corresponding to the erased coefficients, $D_p$ is an $N\times N $ diagonal matrix with  diagonal elements $d_{ii}= q_i$ for $i\in \Lambda$ and 0 otherwise.\\
 The maximum error for a given frame $F$ and it's dual $G$ for m-erasures is defined by 
 $$ \textit{max} \bigg\{ \left\| \Theta_{G}^*D_{p}\Theta_{F}  \right\|: D_p \in \mathcal{D}_{p}^m  \bigg\} $$
where $\mathcal{D}_{p}^m $ is the set of all diagonal matrices with `m' nonzero entries($q_i$ in i'th position) and  zeroes in $N-m$ entries on the main diagonal. The goal is to characterize $(F,G)$ dual pair for a given weight number sequence $\displaystyle{\{q_i\}_{i=1}^N}$ which gives the minimum error and to characterize optimal dual frame $G$ for a preselected frame $F$ with weight number sequence $\{q_i\}_{i=1}^N.$ \\

\begin{defn}
A  dual $G = \{g_i\}_{i=1}^N $ of a frame  $F = \{f_i\}_{i=1}^N$ in a Hilbert space H is called \textit{probabilistic 1-uniform dual of $F$} if $\langle f_i, g_i \rangle  = \frac{1}{q_i} $ for all $1 \leq i \leq N .$ \\
A dual pair $(F,G)$ is called \textit{probabilistic 1-uniform dual pair} if $\langle f_i, g_i \rangle  = \frac{1}{q_i} $ for all $1 \leq i \leq N .$
 \end{defn}

 \begin{defn}
   A dual $G = \{g_i\}_{i=1}^N $ of  a frame  $F = \{f_i\}_{i=1}^N$ in a Hilbert space H is called \textit{probabilistic 2-uniform dual of $F$} if  $\langle f_i ,g_j \rangle \langle f_j , g_i \rangle = \frac{1}{q_i q_j}$ ,\;  $1\leq i \neq j \leq N.$\\
 A dual pair $(F,G)$ is called \textit{probabilistic 2-uniform dual pair} if  $\langle f_i ,g_j \rangle \langle f_j , g_i \rangle = \frac{1}{q_i q_j} ,$  \; $1\leq i \neq j \leq N.$\\
 \end{defn}
 
 \begin{defn}
 A $(N,n)$ parseval frame  $F = \{f_i\}_{i=1}^N$ in $H$ with weight number sequence  $\{q_i\}_{i=1}^N$ given by  (\ref{eqn2point1}) is called \textit{probability uniform parseval frame } if it satisfy \;$\| f_i \|^2 = \frac{1}{q_i}$, for all $ 1\leq i \leq N.$
 \end{defn}

 \section{Probabilistic optimal dual and optimal dual pair under spectral norm }
 
Let $F=\{f_i\}_{i=1}^N$ be a frame for $H$ and $G=\{g_i\}_{i=1}^N $ be a dual frame for $F$. Let   $\{q_i\}_{i=1}^N$ be a weight number sequence given by  (\ref{eqn2point1}) . Here $\rho(E_{\Lambda})$ is denoted as spectral radius of  the error operator $E_{\Lambda}$. For each m, $1 \leq m \leq N$  let \\
  $$ \mathcal{R}_{m}^p (F,G) :=\textit{max} \bigg\{\rho\left(\theta_{G}^*D_{p}\theta_{F}\right) : D_{p} \in \mathcal{D}_{m}^p \bigg\} $$  and  


$$ \nu_{m}^p(F) := \textit{inf} \bigg\{ \mathcal{R}_{m}^p (F,G) : \textit{G is an (m-1)-erasure probabilistic spectrally optimal dual of $F$} \bigg\} $$ 
A dual frame $G$ is called \textit{m-erasure probabilistic spectrally optimal dual  of $F$ } if $ \mathcal{R}_{m}^p (F,G) = \nu_{m}^p(F) $.

Now let us define :
\begin{align*}
 & \zeta_{1}^p := \textit{inf} \bigg\{ \mathcal{R}_{1}^p (F,G) : \textit{(F,G) is a dual pair in H} \bigg\}\\
 &\Delta_{1}^p := \bigg\{ (F,G) : \mathcal{R}_{1}^p (F,G) = \zeta_{1}^p \bigg\}\\
&\zeta_{m}^p := \textit{inf} \bigg\{ \mathcal{R}_{m}^p (F,G) : (F,G) \in \Delta_{m-1}^p \bigg\}\\
& \Delta_{m}^p := \bigg\{(F,G) : \mathcal{R}_{m}^p (F,G) = \zeta_{m}^p  \bigg\}
\end{align*}

\noindent
The set of elements in $\Delta_{m}^p$ is called  \textit{ m-erasure probabilistic spectrally optimal dual pair}.

\vskip 1em
\begin{prop}
Let $F=\{f_i\}_{i=1}^N$ be a frame and  $G=\{g_i\}_{i=1}^N$ be a dual frame of F. Let $\{q_i\}_{i=1}^N$ be a weight number sequence is given by  (\ref{eqn2point1}). Then
 $$ \mathcal{R}_{1}^p (F,G) = \textit{max} \bigg\{ q_{i}|\langle f_i , g_i \rangle | : 1\leq i \leq N \bigg\} $$
\end{prop}

\begin{proof}
For a dual pair $(F,G),$ if the error occurs in i'th position, then it is easy to calculate that 
$$\rho\left(\theta_{G}^*D_{p}\theta_{F}\right) = \rho\left(\theta_{F}^*D_{p}\theta_{G}\right) =  q_{i}|\langle f_i , g_i \rangle | $$
Therefore ,\;\; $ \mathcal{R}_{1}^p (F,G) = \textit{max} \bigg\{ q_{i}|\langle f_i , g_i \rangle | : 1\leq i \leq N \bigg\} .$
\end{proof} \hfill{$\square$}

In the following proposition, we give an sufficient condition under which the canonical dual is an 1-erasure probabilistic spectrally optimal dual of a given frame $F$ with weight number sequence $\{q_i\}_{i=1}^N.$\\~\\
\noindent
 Let $F = \{f_i\}_{i=1}^N $ be a frame for $H$ and  let $\{q_i\}_{i=1}^N$ be a weight number sequence given by  (\ref{eqn2point1}). Set  $c =  \textit{max} \bigg\{ q_i \|S_{F}^{-\frac{1}{2}}f_i \|^2 : 1 \leq i \leq N \bigg\}$.  Set $ \Upsilon_1 = \bigg\{ i : q_i \|S_{F}^{-\frac{1}{2}}f_i \|^2 = c \bigg\} $ and $ \Upsilon_2 = \{1,2,...,N\} \setminus \Upsilon_1 $. Set $H_j = \textit{span} \left\{ f_i : i \in \Upsilon_j \right\}$; \;\;for $j=1,2$.
 
\vskip 1em
\begin{prop}\label{prop3point2}
Let $F = \{f_i\}_{i=1}^N $ be a $(n,N)$ frame for $H$ and $\{q_i\}_{i=1}^N$ be a weight number sequence given by  (\ref{eqn2point1}) . If $H_1 \cap H_2 = \{0\}$, then the canonical dual is an 1-erasure probabilistic spectrally optimal dual of F. 
\end{prop}

\begin{proof}
Assume that  $G = \{g_i\}_{i=1}^N = \{S_{F}^{-1}f_i + u_i \}_{i=1}^N$ \;be an 1-erasure probabilistic spectrally optimal dual of $F$, where $\{u_i\}_{i=1}^N$ satisfy the equation :
$$\displaystyle{ \sum_{i=1}^N \langle f,u_i \rangle f_i = 0} \;\;\;\;\; \textit{for all }\;f \in H$$
This can be rewritten as 
$$ \sum_{i \in  \Upsilon_1 } \langle f,u_i \rangle f_i +  \sum_{i \in  \Upsilon_2 } \langle f,u_i \rangle f_i = 0 \;\;\;\;\; \textit{for all }\;f \in H.$$
By assumption 
$$  \sum_{i \in  \Upsilon_1 } \langle f,u_i \rangle f_i =0 \;\;\; \textit{and}\;\;\;  \sum_{i \in  \Upsilon_2 } \langle f,u_i \rangle f_i = 0 \;\;\;\;\textit{for all }\;f \in H .$$
\noindent
This implies, $$\Theta_{\overline{F}}^* \Theta_{\overline{U}} = 0 \;\;\; \textit{where} \; \overline{F} = \{f_i\}_{i \in  \Upsilon_1}\;\; \textit{and} \;\; \overline{U} = \{u_i\}_{ i \in  \Upsilon_1}\;.$$
\noindent
Therefore  $$ \mathrm{Tr}\;\left(\Theta_{\overline{F}}^* \Theta_{\overline{U}}\right) = \sum_{i \in  \Upsilon_1 } \langle f_i , u_i \rangle = 0$$
\noindent
Similarly we can show that 
$$ \sum_{i \in  \Upsilon_2 } \langle f_i , u_i \rangle = 0$$
Therefore
$$ \displaystyle{\textit{max}_{i=1}^N \;q_i |\langle g_i , f_i \rangle |\; \leq \;  \textit{max}_{i=1}^N \;q_i |\langle S_{F}^{-1}f_i , f_i \rangle | \;=\;  \textit{max}_{i=1}^N \; q_i \| S_{F}^{-\frac{1}{2}}f_i  \|^2 \;=\;c }$$
Which implies
$$ q_i |\langle g_i , f_i \rangle | \leq c = q_i |\langle S_{F}^{-1}f_i , f_i \rangle | \;\;\; \forall \; i\in \Upsilon_1 .$$

\noindent
Hence for all $i\in  \Upsilon_1 $ we have \;\;
$ \left| \langle f_i , S_{F}^{-1}f_i \rangle + \langle f_i , u_i \rangle \right| \leq \left| \langle f_i , S_{F}^{-1}f_i \rangle \right| .$\\
Explicitly we can write
\begin{eqnarray}\label{equation2}
\bigg( \langle f_i , S_{F}^{-1}f_i \rangle + \mathrm{Re}\left( \langle f_i , u_i \rangle \right) \bigg)^2 + \bigg( \mathrm{Im} \langle f_i , u_i \rangle \bigg)^2 \leq  \left\langle f_i , S_{F}^{-1}f_i \right\rangle ^2 
\end{eqnarray}
From the equation (\ref{equation2}) for all \;$ i\in \Upsilon_1,$\;\; $ \left|  \|S_{F}^{-\frac{1}{2}}f_i \|^2 +\; re\;(\langle f_i , u_i \rangle) \right| \leq  \|S_{F}^{-\frac{1}{2}}f_i \|^2 .$  This implies that $ \mathrm{Re}\left(\langle f_i , u_i \rangle \right) \leq 0  \;\;\; \textit{for all}\;\; i \in  \Upsilon_1$. Using the fact  $ \sum_{i \in  \Upsilon_1 } \langle f_i , u_i \rangle = 0\;\left( \mathrm{Re} \left(\sum_{i \in  \Upsilon_1 } \langle f_i , u_i \rangle \right) = 0 \right)$, we can conclude that $\mathrm{Re}\langle f_i , u_i \rangle = 0 \;\; \forall i \in \Upsilon_1 .$  And therefore by (\ref{equation2})\;\;
 $\bigg( \mathrm{Im} \langle f_i , u_i \rangle \bigg)^2 = 0 \;\;\; \forall i \in \Upsilon_1$. Hence\;\; $\mathrm{Im} \langle f_i , u_i \rangle = 0 \;\;  \forall i \in \Upsilon_1.$\\
  
  \noindent
 Therefore for all $i \in \Upsilon_1,$\;\; $ \langle f_i , u_i \rangle = 0.$
 
 \noindent
 \begin{eqnarray*}
 \mathcal{R}_{1}^p (F,G) &=& \textit{max} \bigg\{ \textit{max}_{i \in \Upsilon_1}\; q_i |\langle g_i , f_i \rangle | \;\;,\;\; \textit{max}_{i \in \Upsilon_2}\; q_i |\langle g_i , f_i \rangle |\bigg\} \\&=& \textit{max} \bigg\{ \textit{max}_{i \in \Upsilon_1}\; q_i |\langle  S_{F}^{-1}f_i , f_i \rangle | \;\;,\;\; \textit{max}_{i \in \Upsilon_2}\; q_i |\langle g_i , f_i \rangle |\bigg\} \\&=& \textit{max} \bigg\{ c \;\;,\;\; \textit{max}_{i \in \Upsilon_2}\; q_i |\langle g_i , f_i \rangle |\bigg\} \\&\geq& \mathcal{R}_{1}^p (F,S_{F}^{-1}F)
\end{eqnarray*}
 And consequently \;\; $\mathcal{R}_{1}^p (F,G) = \mathcal{R}_{1}^p (F,S_{F}^{-1}F ) $.\\~\\
 Thus the canonical dual is an 1-erasure probabilistic spectrally optimal dual of F.
\end{proof} \hfill{$\square$}\\~\\

\begin{thm}
Let $H$ be a Hilbert space of dimension n and $\{q_i\}_{i=1}^N$ be a weight number sequence given by  (\ref{eqn2point1}). Then $$\zeta_{1}^p = 1 \;\; and\;\;  \Delta_{1}^p = \bigg\{ (F,G) : \langle f_i , g_i \rangle = \frac{1}{q_i}, \forall 1 \leq i \leq N\bigg\}.$$
\end{thm}
\begin{proof}
For a $(n,N)$ dual pair $(F,G)$ in $H$, \;
  $\zeta_{1}^p = \displaystyle{\textit{inf}_{(F,G)}\; \bigg\{ \textit{max}_{i=1}^N \; q_i |\langle g_i , f_i \rangle | \bigg\}}.$ \\
  If $q_i |\langle g_i , f_i \rangle | = s, 1 \leq i \leq N$ and\; for some constant s. Then $$ n = \displaystyle{\sum_{i=1}^N \langle g_i , f_i \rangle \leq \sum_{i=1}^N |\langle g_i , f_i \rangle | = s \sum_{i=1}^N \frac{1}{q_i} =sn }.$$
  So $s \geq 1.$ If we take $s=1$, then $\langle g_i , f_i \rangle = \frac{1}{q_i}$ \; for all $1 \leq i \leq N.$ And therefore $\mathcal{R}_{1}^p (F,G) = 1.$\\
  For any$(n,N)$ dual pair $(F' , G')$, if  $\mathcal{R}_{1}^p (F',G') < 1$\; then $ q_i |\langle g'_i , f'_i \rangle | < 1 ,$ for all $1 \leq i \leq N.$\\
  This implies  $n= \displaystyle{\sum_{i=1}^N \langle g'_i , f'_i \rangle \leq \sum_{i=1}^N |\langle g'_i , f'_i \rangle | < \sum_{i=1}^N \frac{1}{q_i} =n },$ \;\; which is not possible.\\
  Therefore  $\mathcal{R}_{1}^p (F',G') \geq 1,$ for any $(F',G')$ dual pair.\\
  Hence $\zeta_{1}^p = inf \left\{ \mathcal{R}_{1}^p (F,G) : \textit{(F,G) is a dual pair } \right\} = 1. $\\
  If $(F,G) \in \Delta_{1}^p $ then $\textit{max}_{i=1}^N \; q_i|\langle g_i , f_i \rangle | = 1.$ If for any $j \in \{1,2..,N\}\;$ \;\;$q_j|\langle g_j , f_j \rangle | < 1,$ then \newline $ n = \displaystyle{\sum_{i=1}^N \langle g_i , f_i \rangle \leq \sum_{i=1}^N |\langle g_i , f_i \rangle | < \sum_{i=1}^N \frac{1}{q_i} = n } ,$ which is not possible.\\ Therefore  $\langle g_i , f_i \rangle = |\langle g_i , f_i \rangle | = \frac{1}{q_i} $ \; for all $1 \leq i \leq N.$ Hence the result follows.
   
  \end{proof}  \hfill{$\square$}

Now we are going to find the value of $ \mathcal{R}_{2}^p (F,G)$. We also try to characterise probabilistic spectrally  optimal dual frame for two-erasure for a given frame $F=\{ f_i \}_{i=1}^N$ and establish a relation between 1-erasure and 2-erasure. Lastly, we will  characterise probabilistic spectrally  optimal dual pair for 2-erasure.

 For a  $(n,N)$-dual pair $(F,G)$ in H 
 $$ \displaystyle{\mathcal{R}_{2}^p (F,G) = \textit{max}_{i \neq j} \left\{ \rho(\theta_{G}^*D\theta_{F} ) : D \in \mathcal{D}_{2}^{P} \right\} = \textit{max}_{i \neq j} \left\{ \rho(D\theta_{F}\theta_{G}^*) : D \in \mathcal{D}_{2}^{P} \right\}} $$
 It can be easily seen that 
 \begin{equation} \label{eqn3point3}
  \mathcal{R}_{2}^p (F,G) = \textit{max}_{i\neq j} \left| \frac{q_i \alpha_{ii} +q_j \alpha_{jj} \pm \sqrt{(q_i \alpha_{ii} - q_j \alpha_{jj})^2 + 4 q_i q_j\alpha_{ij}\alpha_{ji} }}{2} \right|  \;\;\;\; where \;\;\alpha_{ij} = \langle g_i , f_j \rangle. 
 \end{equation}
 Therefore if $(F,G) \in \Delta_{1}^p ,$ then
 \begin{eqnarray}\label{eqn3point4}
 \mathcal{R}_{2}^p (F,G) &=& \textit{max}_{i\neq j} \left| 1 \pm \sqrt{q_i q_j\alpha_{ij}\alpha_{ji}} \right| \nonumber\\&=&  \textit{max}_{i\neq j} \left| 1 + \sqrt{q_i q_j\alpha_{ij}\alpha_{ji}} \right|
  \end{eqnarray}\\~\\

\noindent
The following theorem establishes a relation between one and two erasure for a given dual pair $(F,G)$ with weight number sequence $\{ q_i \}_{i=1}^N.$

\begin{thm} \label{thm3point2}
Let $F= \{f_i\}_{i=1}^N$ be a $(n,N)$ frame for the Hilbert space H and $\{q_i\}_{i=1}^N$ be a weight number sequence given by  (\ref{eqn2point1}) . let $G= \{g_i\}_{i=1}^N$ be a dual of $F$ satisfying :
\begin{enumerate}
    \item [{\em (i)}] $\langle g_i , f_i \rangle \geq 0  \;\;\;\;\;\;\; \textit{for all}\;\; 1\leq i \leq N$
    \item [{\em (ii)}] $q_iq_j\langle g_j , f_i \rangle \langle g_i , f_j \rangle = c > 0  \;\;\;\;\;\;\; \textit{for all}\;\; i \neq j$
\end{enumerate}
Let\;\; $\zeta = \bigg\{ i: q_i\langle g_i , f_i \rangle  = \mathcal{R}_{1}^p (F,G)\bigg\}$

\noindent
If $| \zeta | =1$, then 
$$ \mathcal{R}_{2}^p (F,G) = \frac{1}{2} \left\{ \mathcal{R}_{1}^p (F,G) + \textit{max}_{i \notin \zeta } q_i\langle g_i , f_i \rangle + \sqrt{\left( \mathcal{R}_{1}^p (F,G)  - \textit{max}_{i \notin \zeta } q_i\langle g_i , f_i \rangle \right)^2 + 4c} \;\right\}$$

\noindent
If $| \zeta | > 1$, then 
$$ \mathcal{R}_{2}^p (F,G) = \mathcal{R}_{1}^p (F,G) + \sqrt{4c} $$
\end{thm}

\begin{proof}
If $| \zeta | = 1$ \\

Let $\Gamma := \{i : max_{j \in \zeta^c} \;\; q_j\langle g_j , f_j \rangle  = q_i\langle g_i , f_i \rangle  \}$. Let $ p\in \Gamma.$\\

As the function $x\mapsto  \left( x +q_j\alpha_{jj} + \sqrt{(x -q_j\alpha_{jj})^{2}+4c} \right )$ is an increasing function of $x$ and hence using the equation (\ref{eqn3point3}) we can rewrite $ \mathcal{R}_{2}^p (F,G) $ as
\begin{eqnarray}\label{eqn3point5}
  \mathcal{R}_{2}^p (F,G) = \textit{max}_{i \neq k} \;\frac{1}{2} \left\{  \mathcal{R}_{1}^p (F,G) + q_i \alpha_{ii} + \sqrt{\left(  \mathcal{R}_{1}^p (F,G) -  q_i \alpha_{ii}\right)^2 +4c} \;\right\}
\end{eqnarray}

\noindent  
It is suffices to show that 
$$ \mathcal{R}_{2}^p (F,G) = \frac{1}{2} \left\{  \mathcal{R}_{1}^p (F,G) + q_p \alpha_{pp} + \sqrt{\left(  \mathcal{R}_{1}^p (F,G) -  q_p \alpha_{pp}\right)^2 +4c}\;\right\} $$

\noindent
If the maximum occurs in (\ref{eqn3point5}) for some $i=s,$  where $s \notin \zeta, \Gamma$ \;then we have 
$$  \frac{1}{2} \left\{  \mathcal{R}_{1}^p (F,G) + q_p \alpha_{pp} + \sqrt{\left(  \mathcal{R}_{1}^p (F,G) -  q_p \alpha_{pp}\right)^2 +4c}\; \right\} <  \frac{1}{2} \left\{  \mathcal{R}_{1}^p (F,G) + q_s \alpha_{ss} + \sqrt{\left(  \mathcal{R}_{1}^p (F,G) -  q_s \alpha_{ss}\right)^2 +4c}\;\right\} $$

\noindent
This implies $$ \sqrt{\left(  \mathcal{R}_{1}^p (F,G) -  q_s \alpha_{ss}\right)^2 +4c} - \sqrt{\left(  \mathcal{R}_{1}^p (F,G) -  q_p \alpha_{pp}\right)^2 +4c} \;\;> \;\; q_p \alpha_{pp} - q_s \alpha_{ss} \;>\; 0 $$
Squaring we have 
\begin{align*} 
  & \left(  \mathcal{R}_{1}^p (F,G) -  q_s \alpha_{ss}\right)^2 +\left(  \mathcal{R}_{1}^p (F,G) -  q_p \alpha_{pp}\right)^2 +8c - 2\sqrt{\left(  \mathcal{R}_{1}^p (F,G) -  q_s \alpha_{ss}\right)^2 +4c} \; \sqrt{\left(  \mathcal{R}_{1}^p (F,G) -  q_p \alpha_{pp}\right)^2 +4c} \\& >  \left(  \mathcal{R}_{1}^p (F,G) -  q_s \alpha_{ss}\right)^2 + \left(  \mathcal{R}_{1}^p (F,G) -  q_p \alpha_{pp}\right)^2 - 2\left(  \mathcal{R}_{1}^p (F,G) -  q_s \alpha_{ss}\right)\left(  \mathcal{R}_{1}^p (F,G) -  q_p \alpha_{pp}\right) 
 \end{align*}
which gives
$$ 4c + \left(  \mathcal{R}_{1}^p (F,G) -  q_s \alpha_{ss}\right) \left(  \mathcal{R}_{1}^p (F,G) -  q_p \alpha_{pp}\right) \;\;>\;\; \sqrt{\left(  \mathcal{R}_{1}^p (F,G) -  q_s \alpha_{ss}\right)^2 +4c}\;\sqrt{\left(  \mathcal{R}_{1}^p (F,G) -  q_p \alpha_{pp}\right)^2 +4c} \;\;>\;\; 0 $$

\noindent
again squaring we have 
\begin{align*} 
     & \left(  \mathcal{R}_{1}^p (F,G) -  q_s \alpha_{ss}\right)^2 \left(  \mathcal{R}_{1}^p (F,G) -  q_p \alpha_{pp}\right)^2 +16c^2 + 8c\left(  \mathcal{R}_{1}^p (F,G) -  q_s \alpha_{ss}\right) \left(  \mathcal{R}_{1}^p (F,G) -  q_p \alpha_{pp}\right) \;\;  \\&>  \left(  \mathcal{R}_{1}^p (F,G) -  q_s \alpha_{ss}\right)^2 \left(  \mathcal{R}_{1}^p (F,G) -  q_p \alpha_{pp}\right)^2 + 4c\bigg\{ \left(  \mathcal{R}_{1}^p (F,G) -  q_s \alpha_{ss}\right)^2 + \left(  \mathcal{R}_{1}^p (F,G) -  q_p \alpha_{pp}\right)^2 \bigg\} +16c^2 \\
     & \implies 2\left(  \mathcal{R}_{1}^p (F,G) -  q_s \alpha_{ss}\right)\left(  \mathcal{R}_{1}^p (F,G) -  q_p \alpha_{pp}\right) \;>\; \left(  \mathcal{R}_{1}^p (F,G) -  q_s \alpha_{ss}\right)^2 + \left(  \mathcal{R}_{1}^p (F,G) -  q_p \alpha_{pp}\right)^2\\
     & \implies  \left( q_p \alpha_{pp} - q_s \alpha_{ss} \right)^2 \;<\; 0
   \end{align*}

\noindent
 This is a contradiction.\\

\noindent
If $| \zeta | > 1.$ let $i_1,i_2 \in \zeta.$\\
Then\;\; $ \mathcal{R}_{1}^p (F,G) = q_{i_1} | \langle g_{i_1} , f_{i_1} \rangle| = q_{i_2} | \langle g_{i_2} , f_{i_2} \rangle|. $\\~\\
Using (\ref{eqn3point3}) we can express $ \mathcal{R}_{2}^p (F,G) $ as 
\begin{eqnarray}\label{eqn3point6}
\mathcal{R}_{2}^p (F,G) = \textit{max}_{i \neq i_1} \;\frac{1}{2} \left\{  q_{i_1} \alpha_{i_1i_1} + q_{i}\alpha_{ii} + \sqrt{\left(  q_{i_1} \alpha_{i_1i_1} - q_{i}\alpha_{ii}\right)^2 +4c}\;\right\} 
\end{eqnarray}

\noindent
It suffices to show that the maximum on the right hand side of equation (\ref{eqn3point6}) is attained when $i=i_2.$\\
 For $i=i_2$ the expression on the right hand side of  (\ref{eqn3point6}) become $\mathcal{R}_{1}^p (F,G) + \sqrt{4c}.$

\noindent 
 If the maximum of the expression  (\ref{eqn3point6}) attained for some $k \notin \zeta$, then we have 
 $$\frac{1}{2} \bigg\{  q_{i_1} \alpha_{i_1i_1} + q_{k}\alpha_{kk} + \sqrt{\left(  q_{i_1} \alpha_{i_1i_1} - q_{k}\alpha_{kk}\right)^2 +4c}\bigg\} \geq  \mathcal{R}_{1}^p (F,G) + \sqrt{4c} $$
 
  i.e., $$0 < \mathcal{R}_{1}^p (F,G) - q_{k}\alpha_{kk} + \sqrt{4c} \leq  \sqrt{\left(  q_{i_1} \alpha_{i_1i_1} - q_{k}\alpha_{kk}\right)^2 +4c} $$
  
  Squaring the above relation on both sides and doing some simple rearrangements we get 
  $$ 4\sqrt{c}\left(\mathcal{R}_{1}^p (F,G) - q_{k}\alpha_{kk} \right)^2 \leq 0 $$ 
  This is a contradiction as  $k \notin \zeta.$
  Hence the result follows. \hfill{$\square$}
\end{proof}

\noindent
If a dual pair $(F,G)$ is a probabilistic spectrally  1-erasure optimal and satisfy the condition of theorem (\ref{thm3point2}) then it is also 2-erasure probabilistic spectrallly optimal dual pair.\\
\begin{cor}
  Let $F = \{f_i\}_{i=1}^N $ be a frame for $H$ and  Let $\{q_i\}_{i=1}^N$ be a weight number sequence given by  (\ref{eqn2point1}) . Set  $c =  \textit{max} \left\{ q_i \|S_{F}^{-\frac{1}{2}}f_i \|^2 : 1 \leq i \leq N \right\}$.  Set $ \Upsilon_1 = \left\{ i : q_i \|S_{F}^{-\frac{1}{2}}f_i \|^2 = c \right\} $ and $ \Upsilon_2 = \{1,2,...,N\} \setminus \Upsilon_1 $. Set $H_j = \textit{span} \left\{ f_i : i \in \Upsilon_j \right\}$; \;\;for $j=1,2$.\\
  If $H_1 \cap H_2 ={0}, $ $|\Upsilon_1| \geq 2$\;\;and $\left|\langle S_{F}^{-\frac{1}{2}}f_i , S_{F}^{-\frac{1}{2}}f_j \rangle \right| = \sqrt{\frac{1}{q_i q_j} \cdot \dfrac{n- \displaystyle{\sum_{i=1}^N \frac{1}{q_{i}^2}}}{\displaystyle{\sum_{r \neq s}\frac{1}{q_r q_s}}}}$ \;\; for all $i \neq j.$ Then the canonical dual of $F$ is 2-erasure probabilistic spectrally optimal dual of $F.$
\end{cor}
\begin{proof}
From the proposition (\ref{prop3point2}),  canonical dual of is an 1-erasure probabilistic spectrally optimal dual of $F.$\\
The value of $\mathcal{R}_{1}^p(F, S_{F}^{-1}F)$ \; attains for more than one $i$.\\
Now 
\begin{eqnarray*}
 q_i q_j \langle S_{F}^{-1}f_j , f_i \rangle \langle S_{F}^{-1}f_i ,f_j \rangle =  q_i q_j \left| \langle S_{F}^{-\frac{1}{2}}f_i , S_{F}^{-\frac{1}{2}}f_j \rangle  \right|^2 = \frac{n- \displaystyle{\sum_{i=1}^N \frac{1}{q_{i}^2}}}{\displaystyle{\sum_{r \neq s}\frac{1}{q_r q_s}}}
\end{eqnarray*}
Therefore by using theorem (\ref{thm3point2}), we can say that the canonical dual  is two-erasure probabilistic spectrally optimal dual of $F.$  \hfill{$\square$}
\end{proof}

\noindent
Now we characterize all such dual pairs $(F,G)$ which attain the optimal value.

\begin{thm}
Let $H$ be an Hilbert space of dimension n and $N \geq n.$  Let $\{q_i\}_{i=1}^N$ be a weight number sequence given by  (\ref{eqn2point1}) . Let $\beta = n - \displaystyle{\sum_{i=1}^N \frac{1}{q_{i}^2}}.$  Then  \\~\\
$$
\zeta_{2}^p=\begin{cases}
			1 + \sqrt{\dfrac{n - \displaystyle{\sum_{i=1}^N \frac{1}{q_{i}^2}}}{\displaystyle{\sum_{i \neq j}\frac{1}{q_iq_j}}}}\;\;, & \text{if $\beta \geq 0$ }\\~\\
            \sqrt{\dfrac{n^2 - n}{n^2 - \displaystyle{\sum_{i=1}^N \frac{1}{q_{i}^2}}}}\;\;\;\;, & \text{if $\beta < 0$}
		 \end{cases}
$$$$  $$
and $$  \Delta_{2}^p = \left\{ (F,G) \in  \Delta_{1}^p  : q_i q_j \alpha_{ij}\alpha_{ji} = \frac{n - \displaystyle{\sum_{i=1}^N \frac{1}{q_{i}^2}}}{\displaystyle{\sum_{i \neq j}\frac{1}{q_iq_j}}} \;\;\;\forall i\neq j \right\}$$
 where $\alpha_{ji} = \langle g_i , f_j \rangle $.
\end{thm}

\begin{proof}
First consider $\beta \geq 0.$\\
Let$(F,G) $ be a $(n,N)$ dual pair in $H$ with weight number sequence  $\{q_i\}_{i=1}^N$ such that $ (F,G) \in \Delta_{1}^p.$\\
Using the expression (\ref{eqn3point3}), we can write $\mathcal{R}_{2}^p (F,G)$ as
 \begin{eqnarray}\label{eqn3point7}
 \mathcal{R}_{2}^p (F,G) = \textit{max}_{i \neq j} \left| 1+ \sqrt{q_i q_j \alpha_{ij}\alpha_{ji}} \right|
 \end{eqnarray}
Let us consider $q_i q_j \alpha_{ij}\alpha_{ji} = c \;\; \textit{for all}\;\; i \neq j, $ and for some constant c.\\
It is easy to see that $\displaystyle{\sum_{i,j = 1}^N \alpha_{ij}\alpha_{ji} = \sum_{i=1}^N \bigg\langle g_i , \sum_{j=1}^N \langle f_i , g_j \rangle f_j \bigg\rangle = \sum_{i=1}^N \langle f_i , g_i \rangle = n}.$\\

\noindent
Using the fact  $ (F,G) \in \Delta_{1}^p,$ it can be easily obtain that $\displaystyle{\sum_{i\neq j}\alpha_{ij}\alpha_{ji}} = n -\displaystyle{\sum_{i=1}^N \frac{1}{q_{i}^2}} = \beta .$\\
As  $\alpha_{ij}\alpha_{ji} = \dfrac{c}{q_i q_j} \;\;\; \forall i \neq j,$\; summing we get \;\; $ \beta = c \displaystyle{\sum_{i \neq j}\frac{1}{q_iq_j}} .$\\
Therefore $\alpha_{ij}\alpha_{ji} = \dfrac{\beta}{q_i q_j \displaystyle{\sum_{i \neq j}\frac{1}{q_i q_j}}} \;\; \textit{for all}\;\; i \neq j.$\\
From (\ref{eqn3point7}) we can say that
$$\mathcal{R}_{2}^p (F,G) = 1 + \sqrt{c} = 1+ \sqrt{\dfrac{\beta}{\displaystyle{\sum_{i \neq j}\frac{1}{q_i q_j}}}} $$
\noindent
Let $(F' ,G') \in \Delta_{1}^p ,$ \; where $F' = \{f'_{i} \}_{i=1}^N$ and  $G' = \{g'_{i} \}_{i=1}^N.$\\

claim : $ max_{i \neq j} \;\;q_i q_j \alpha'_{ij} \alpha'_{ji} \geq c$,\;\;\; where $\alpha'_{ij} = \langle g'_i, f'_j \rangle.$\\~\\
If not, then
   $$ q_i q_j \alpha'_{ij} \alpha'_{ji} < c \;\;\;\forall i \neq j $$
   Summing over $i \neq j$ we get 
   \begin{align*}
& \sum_{i\neq j} \alpha'_{ij} \alpha'_{ji} < c \sum_{i \neq j} \frac{1}{q_i q_j} \\
& i.e, \;\; \beta < \beta 
   \end{align*} 
which is not possible.
Therefore $\mathcal{R}_{2}^p (F,G) \leq \mathcal{R}_{2}^p (F',G').$\\

\noindent
Hence $(F,G)$ dual pair gives the probabilistic two-erasure spectral optimality and 
$$\mathcal{R}_{2}^p (F,G) = 1 + \sqrt{c} = 	1 + \sqrt{\dfrac{n - \displaystyle{\sum_{i=1}^N \frac{1}{q_{i}^2}}}{\displaystyle{\sum_{i \neq j}\frac{1}{q_iq_j}}}} .$$\\

\noindent
If $\beta < 0.$\\
Let $(F,G) \in \Delta_{1}^p,$\;\; which satisfy $q_i q_j \alpha_{ij}\alpha_{ji} = c \;\; \textit{for all}\;\;i \neq j.$\\
It easily follows that $c = \dfrac{\beta}{\displaystyle{\sum_{i \neq j}\frac{1}{q_iq_j}}}.$\\
Then 
$$ \mathcal{R}_{2}^p (F,G) = | 1 + i\sqrt{c}| = \sqrt{1 - c} = \sqrt{1 - \dfrac{\beta}{\displaystyle{\sum_{i \neq j}\frac{1}{q_iq_j}}} } = \sqrt{\dfrac{n^2 - n}{n^2 - \displaystyle{\sum_{i=1}^N \frac{1}{q_{i}^2}}}}$$

\noindent
Let $(F' ,G') \in \Delta_{1}^p ,$ \; where $F' = \{f'_{i} \}_{i=1}^N$ and  $G' = \{g'_{i} \}_{i=1}^N.$\\
Then $\mathcal{R}_{2}^p (F',G') = \textit{max}_{i \neq j} \;\;\left| 1 + \sqrt{q_i q_j \alpha'_{ij}\alpha'_{ji}} \right|. $\\

\noindent
 $\mathcal{R}_{2}^p (F',G') < \mathcal{R}_{2}^p (F,G) \implies
 \textit{max}_{i \neq j} \;\;\left| 1 + \sqrt{q_i q_j \alpha'_{ij}\alpha'_{ji}} \right| < \sqrt{1 - c} \implies \left| 1 + \sqrt{q_i q_j \alpha'_{ij}\alpha'_{ji}} \right| < \sqrt{1 - c}  \;\;\; \forall i \neq j \;\;\; \implies -c > q_i q_j \alpha'_{ij}\alpha'_{ji}  \;\;\; \forall i\neq j.$ \\ 

\noindent
If $q_i q_j \alpha'_{ij}\alpha'_{ji} < 0,$ \; then $\left | q_i q_j \alpha'_{ij}\alpha'_{ji} \right| <  | c| . $

\noindent
If $q_i q_j \alpha'_{ij}\alpha'_{ji} > 0,$ \; then $1 + \sqrt{q_i q_j \alpha'_{ij}\alpha'_{ji}}  < \sqrt{1 + |c|}.$\\
When If $q_i q_j \alpha'_{ij}\alpha'_{ji} > 0,$ if $q_i q_j \alpha'_{ij}\alpha'_{ji} \geq |c|$ happens, then $1+ \sqrt{q_i q_j \alpha'_{ij}\alpha'_{ji}} \geq 1+ \sqrt{|c|} \geq \sqrt{1 + |c|}.$ Which is not possible . Therefore $q_i q_j \alpha'_{ij}\alpha'_{ji} < |c|.$\\~\\
Combining both cases we can conclude that 
$$ \left| q_i q_j \alpha'_{ij}\alpha'_{ji} \right| < |c| = -c$$
i.e.,\;\; $q_i q_j \alpha'_{ij}\alpha'_{ji} > c$\\

\noindent
Therefore $ \alpha'_{ij}\alpha'_{ji} > \frac{c}{ q_i q_j } \;\;\; \textit{for all}\; i \neq j.$ \\
Summing we have 
$$\sum_{i \neq j} \alpha'_{ij}\alpha'_{ji} > c \sum_{i \neq j} \frac{1}{q_i q_j} = \beta $$
\noindent
This gives rise to a contradiction. Therefore $\mathcal{R}_{2}^p (F,G) \leq \mathcal{R}_{2}^p (F',G').$\\

\noindent
Hence $(F,G)$ dual pair gives the probabilistic two-erasure spectral optimality and 
$$\mathcal{R}_{2}^p (F,G) = \sqrt{\frac{n^2 - n}{n^2 - \sum_{i=1}^N \frac{1}{q_{i}^2}}} .$$

\noindent
Combining both cases we can say 
$$
\zeta_{2}^p=\begin{cases}
			1 + \sqrt{\frac{n - \displaystyle{\sum_{i=1}^N \frac{1}{q_{i}^2}}}{\displaystyle{\sum_{i \neq j}\frac{1}{q_iq_j}}}}\;\;, & \text{if $\beta \geq 0$ }\\~\\
            \sqrt{\dfrac{n^2 - n}{n^2 - \displaystyle{\sum_{i=1}^N \frac{1}{q_{i}^2}}}}\;\;, & \text{if $\beta < 0$}
		 \end{cases}
  $$
and $$ \Delta_{2}^p = \bigg\{ (F,G) \in  \Delta_{1}^p  : q_i q_j \alpha_{ij}\alpha_{ji} = c \;\; \textit{for some constant c},\;\; \forall i \neq j \bigg\}$$
In other words
$$ \Delta_{2}^p = \left\{ (F,G) \in  \Delta_{1}^p  : q_i q_j \alpha_{ij}\alpha_{ji} = \dfrac{n - \sum_{i=1}^N \frac{1}{q_{i}^2}}{\sum_{i \neq j}\frac{1}{q_iq_j}} \right\}.  $$ \hfill{$\square$}

\end{proof}

\section{Probabilistic optimal dual and optimal dual pair under operator norm}

Let $F=\{f_i\}_{i=1}^N$ be a frame for $H$ and $G=\{g_i\}_{i=1}^N $ be a dual frame for $F$. Let   $\{q_i\}_{i=1}^N$ be a weight number sequence given by  (\ref{eqn2point1}) . For each $ 1 \leq m \leq N$,  let us define \\
  $$ d_{m}^p (F,G) =\textit{max} \bigg\{\left\|\theta_{G}^*D_{p}\theta_{F}\right\| : D_{p} \in \mathcal{D}_{m}^p \bigg\}. $$ 
  and
 $$ \mu_{1}(F) =inf \; \left\{ d_{1}^p (F,G) : \textit{G is a dual frame of F} \right\} $$

\noindent 
 A dual frame G is called 1-erasure probabilistic  optimal dual of F if\;\; $\mu_{1}(F) = d_{1}^p (F,G) .$
  $$  \mu_{m}(F) =inf \; \left\{ d_{m}^p (F,G) : \textit{G is an (m-1)-erasure probabilistic optimal dual of F} \right\}$$
 
\noindent 
 A dual frame G is called m-erasure probabilistic  optimal dual of F if\;\; $\mu_{m}(F) = d_{m}^p (F,G) .$ 
  \begin{align*}
&\varepsilon_{1}^p := inf \; \bigg\{ d_{1}^p (F,G) : (F,G )\; \textit{is a dual pair in H }\bigg\} \\ 
& \Gamma_{1}^p := \bigg\{ (F,G) \;\;\textit{dual pair} : d_{1}^p (F,G) = \varepsilon_{1}^p \bigg\} \\
&\varepsilon_{m}^p := inf \; \bigg\{ d_{m}^p (F,G) : (F,G ) \in \Gamma_{m-1}^p\bigg\} \\
& \Gamma_{m}^p := \bigg\{ (F,G) \;\;\textit{dual pair} : d_{m}^p (F,G) = \varepsilon_{m}^p \bigg\}  
  \end{align*}

\noindent
Let $F = \{f_i\}_{i=1}^N $ be a frame for $H$ and  let $\{q_i\}_{i=1}^N$ be a weight number sequence given by  (\ref{eqn2point1}) . Let $G = \{g_i\}_{i=1}^N $ be a dual of $F$ in $H.$  It is trivially follows that $d_{1}^p (F,G) =\textit{max}_{i=1}^N \; q_i\| f_i \|\;\| g_i \|. $\\

\noindent
For a frame $F = \{f_i\}_{i=1}^N $ in $H$ and   $\{q_i\}_{i=1}^N$ be a weight number sequence given by  (\ref{eqn2point1}) . Set  $c =  \textit{max} \bigg\{ q_i \| f_i \|\; \|S_{F}^{-1}f_i \| : 1 \leq i \leq N \bigg\}$.  Set $ \varrho_1 = \bigg\{ i :q_i \| f_i \|\; \|S_{F}^{-1}f_i \| = c \bigg\} $ and $ \varrho_2 = \{1,2,...,N\} \setminus \varrho_1 $. Set $H_j = \textit{span} \left\{ f_i : i \in \varrho_j \right\}$; \;\;for $j=1,2$.

\begin{prop}\label{prop4point1}
Let $F = \{f_i\}_{i=1}^N $ be a frame for $H$ and  Let $\{q_i\}_{i=1}^N$ be a weight number sequence given by  (\ref{eqn2point1}). If $H_1 \cap H_2 = \{0\}$, then the canonical dual is an 1-erasure probabilistic optimal dual of F.
\end{prop}
\begin{proof}
Let $G = \{g_i\}_{i=1}^N = \{ S_{F}^{-1}f_i + u_i \}_{i=1}^N $ be an 1-erasure probabilistic optimal dual of $F.$ \\~\\
Then \;\;\; $d_{1}^p (F,G) < d_{1}^p (F,  S_{F}^{-1} F) .$\\
i.e., 
 \begin{align} \label{eqn4point9}
    &\;\;\; max_{i=1}^N \; q_i \|f_i \|\; \|g_i \| \leq max_{i=1}^N \; q_i \|f_i \|\; \|S_{F}^{-1}f_i \| = c \nonumber \\
  &\implies     max_{i \in \varrho_1 } \;\; q_i \|f_i \|\; \|g_i \| \leq  max_{i \in \varrho_1}  \;\;q_i \|f_i \|\; \| S_{F}^{-1}f_i \| \nonumber \\
  &\implies  q_i \|f_i \|\; \|g_i \| \leq q_i \|f_i \|\; \| S_{F}^{-1}f_i \| \;\;\; \forall i \in \varrho_1  \nonumber \\
  &\implies \|g_i \|^2 \leq \| S_{F}^{-1}f_i \|^2  \;\;\; \forall i \in \varrho_1 \nonumber\\
  &\implies \| u_i \|^2 + 2Re \langle S_{F}^{-1}f_i , u_i \rangle \leq 0 \;\;\; \forall i \in \varrho_1
   \end{align}

\noindent
As $G = \{ g_i \}_{i=1}^n $ be a dual of $F$ ,  
 $$  \sum_{i=1}^N \langle  f , u_i \rangle f_i = 0 \;\;\;\; \textit{for all}\; f \in H $$
Using the condition $H_1 \cap H_2 = \{0\},$ we have
  \begin{align} \label{eqn4point10}
  \sum_{i \in \varrho_1} \langle  f , u_i \rangle f_i = 0 =  \sum_{i \in \varrho_2} \langle  f , u_i \rangle f_i \;\;\;\; \textit{for all} \;f \in H .
  \end{align}

\noindent 
Let $U^1 = \{ u_i\}_{ i \in \varrho_1}$ ,  $U^2 = \{ u_i\}_{ i \in \varrho_2}$,  $F^1 = \{ f_i\}_{ i \in \varrho_1}$  , $F^2 = \{ f_i\}_{ i \in \varrho_2}$ \\

\noindent  
From equation (\ref{eqn4point10}) we have
  $$ \Theta_{F^1}^*\Theta_{U^1} = 0 = \Theta_{F^2}^*\Theta_{U^2} $$
 Therefore \;\; $\mathrm{Tr} \left(\Theta_{F^1}^* \Theta_{U^1} \right)  = \mathrm{Tr} \left(\Theta_{U^1} \Theta_{F^1}^* \right) = 0 =  \sum_{i \in \varrho_1} \langle  f_i , u_i \rangle$\;\; 
 and \;\;  $ \mathrm{Tr}\left(\Theta_{F^2}^* \Theta_{U^2} \right)  = \mathrm{Tr} \left(\Theta_{U^2} \Theta_{F^2}^* \right) = 0 =  \sum_{i \in \varrho_2} \langle  f_i , u_i \rangle.$\\
  \noindent 
 Also using the fact \;\; $ S_{F}^{-1} \Theta_{F^1}^*\Theta_{U^1} = \Theta_{S_{F}^{-1} F^1}^*\Theta_{U^1} = 0 $ \;and \; $ S_{F}^{-1} \Theta_{F^2}^*\Theta_{U^2} = \Theta_{S_{F}^{-1} F^2}^*\Theta_{U^2} = 0 $\;\; we have
 $$ \sum_{i \in \varrho_1} \langle  S_{F}^{-1}f_i ,u_i \rangle = 0 =  \sum_{i \in \varrho_2} \langle  S_{F}^{-1}f_i ,u_i \rangle$$
 
 \noindent 
 From (\ref{eqn4point9}), summing we have \;\;$ \sum_{i \in \varrho_1} \| u_i \|^2 + 2 Re \bigg( \sum_{i \in \varrho_1} \langle  S_{F}^{-1}f_i ,u_i \rangle \bigg) \leq 0 .$ This implies 
    $ \sum_{i \in \varrho_1} \| u_i \|^2 \leq 0 .$ And hence $ u_i =0 ; \forall i \in \varrho_1.$\\
  Therefore
  \begin{align*}
      d_{1}^p (F,G) &= max\; \left\{ max_{ i \in \varrho_1} q_i \|f_i \|\; \|g_i \| \;,\; max_{ i \in \varrho_2} q_i \|f_i \|\; \|g_i \| \right\} \\ 
      &=  max\; \left\{ c\;,\;  max_{ i \in \varrho_2} q_i \|f_i \|\; \|g_i \| \right\}\\
      &\geq\; d_{1}^p (F,  S_{F}^{-1}F)
  \end{align*}
Hence $ d_{1}^p (F,G) =  d_{1}^p (F,  S_{F}^{-1}F).$  \hfill{$\square$}
  
\end{proof}

\begin{cor}
 Let $F = \{f_i\}_{i=1}^N $ be a frame for $H$ and  let $\{q_i\}_{i=1}^N$ be a weight number sequence given by  (\ref{eqn2point1}). If $H_1 \cap H_2 = \{0\}$ and $\{f_i : i \in \varrho_2 \} $ are linearly independent, then the canonical dual is the unique  1-erasure probabilistic optimal dual of F and therefore m-erasure optimal dual dual for $F.$
\end{cor}

\begin{proof}
Let $G = \{g_i\}_{i=1}^N = \{ S_{F}^{-1}f_i + u_i \}_{i=1}^N $ be an 1-erasure probabilistic optimal dual of $F.$ \\
From proposition (\ref{prop4point1}), $u_i =0 \;\; \textit{for all} \;\; i \in \varrho_1.$\\
Therefore it is enough to show that $u_i = 0 \;\; \textit{for all} \;\; i \in \varrho_2.$\\
We have\;\; $ \sum_{i=1}^N \langle  f , u_i \rangle f_i = 0 \;\;\;\; \textit{for all}\; f \in H .$\\
Using the fact  $H_1 \cap H_2 = \{0\}$ we have 
$$  \sum_{i \in \varrho_1} \langle  f , u_i \rangle f_i = 0 =  \sum_{i \in \varrho_2} \langle  f , u_i \rangle f_i$$
 As  $\{f_i : i \in \varrho_2 \} $ are linearly independent then
 $\langle  f , u_i \rangle  = 0,$ for all $i \in \varrho_2$ and for any $i \in H.$\\
 This implies $u_i =o \;\; \textit{for all} \;\; i \in \varrho_2.$
 \hfill{$\square$}
\end{proof}

\begin{thm}\label{thm4point1}
Let $H$ be an Hilbert space of dimension n and $N \geq n.$  Let $\{q_i\}_{i=1}^N$ be a weight number sequence given by  (\ref{eqn2point1}) .Then
$$ \varepsilon_{1}^p = 1 $$ and
$$ \Gamma_{1}^p = \bigg\{ (F,G) :\;\; \langle f_i, g_i \rangle = \|f_i\|\;\|g_i \| = \frac{1}{q_i} \;\;\; \forall 1 \leq i \leq N\bigg\} $$
\end{thm}
\begin{proof}
For a frame  $F = \{f_i \}_{i=1}^N$ and it's dual  $G = \{g_i \}_{i=1}^N$, if $ q_i\|f_i\|\;\|g_i \| = c$, for some constant c for all $ 1 \leq i \leq N;$ then
$$ n = \sum_{i=1}^N \langle f_i , g_i \rangle \leq \sum_{i=1}^N \|f_i\|\;\|g_i \| = c \sum_{i=1}^N \frac{1}{q_i} =cn .$$
This implies $c \geq 1.$\\
If $ c=1$ then $\|f_i\|\;\|g_i \| =  \langle f_i , g_i \rangle =  \frac{1}{q_i} $, for all $1 \leq i \leq n .$\\

Claim : For any frame  $F' = \{f'_i \}_{i=1}^N$ and it's dual  $G' = \{g'_i \}_{i=1}^N$ \;\; $d_{1}^p (F' ,G' ) \geq 1 .$\\~\\
If not then \;\;$ q_i\|f'_i\|\;\|g'_i \| < 1  $, for all $1 \leq i \leq n .$ That is \;\; $\|f'_i\|\;\|g'_i \| \leq \frac{1}{q_i},$ \;\; for all  $1 \leq i \leq n .$ This implies  $n = \sum_{i=1}^N \langle f'_i , g'_i \rangle \leq \sum_{i=1}^N \|f'_i\|\;\|g'_i \| < \sum_{i=1}^N  \frac{1}{q_i} = n.$
Which is not possible. Therefore  $d_{1}^p (F' ,G' ) \geq 1 .$\\
Hence $ \varepsilon_{1}^p = 1 .$\\
\noindent
Now $ \Gamma_{1}^p = \bigg\{ (F,G) \;\;\textit{dual pair}:\;\; max_{i=1}^N\;q_i \|f_i\|\;\|g_i \| = 1 \bigg\} .$ \\
If any $q_j \|f_j\|\;\|g_j  \| \leq 1$, $1 \leq j \leq n,$ \; then $n = \sum_{i=1}^N \langle f_i, g_i \rangle \leq \sum_{i=1}^N | \langle f_i, g_i \rangle | \leq \sum_{i=1}^N \|f_i\|\;\|g_i \| < \sum_{i=1}^N \frac{1}{q_i} =n .$ This arise a contradiction.\\
Therefore  $ \Gamma_{1}^p = \bigg\{ (F,G) \;\;\textit{dual pair}:\;\;q_i \langle f_i, g_i \rangle = q_i \|f_i\|\;\|g_i \| = 1 \;\;\textit{for all}\;\; 1 \leq i \leq N\bigg\} .$    \hfill{$\square$}

\end{proof}

\begin{defn}
Let  $\{ p_i\}_{i=1}^N$ be a probability sequence satisfying (\ref{eqn2.1}) and $\{q_i\}_{i=1}^N$ be the corresponding  weight number sequence defined by  (\ref{eqn2point1}). We call a parseval frame $F = \{f_i\}_{i=1}^N $ a \textit{probabilistic uniform parseval frame} if it satisfies $\| f_i \| = \frac{1}{\sqrt{q_i}},$ \; for all $1 \leq i \leq N.$
\end{defn}

\begin{cor}
 Let $F = \{f_i\}_{i=1}^N $ be a probabilistic uniform perseval frame for $H$ and  let $\{q_i\}_{i=1}^N$ be a weight number sequence given by  (\ref{eqn2point1}).  Then it's canonical dual is the unique 1-erasure probabilistic optimal dual of $F$ and therefore m-erasure probabilistic optimal dual of $F.$
\end{cor}

\begin{proof}
As $F = \{f_i\}_{i=1}^N $ is a probabilistic uniform perseval frame for $H$ then $\| f_i \|^2 = \frac{1}{q_i}$ \;\; for all $1 \leq i \leq N .$ \\
It is easy to see that $d_{m}^p (F, S_{F}^{-1}F) = 1$ \; and hence by theorem (\ref{thm4point1}), canonical dual is an 1-erasure probabilistic optimal dual of $F.$\\
Let $G = \{g_i\}_{i=1}^N = \{ f_i + u_i\}_{i=1}^N $ be an 1-erasure probabilistic optimal dual of $F.$\\
Therefore \;\; $max_{i=1}^N\;q_i \|f_i\|\;\| f_i + u_i\| = max_{i=1}^N\;q_i \|f_i\|^2 .$\\
This implies 
  \begin{align}\label{eqn4point11}
     & q_i \|f_i\|\;\| f_i + u_i\| \leq q_i \|f_i\|^2 \;\;\;\;\forall 1 \leq i \leq N \nonumber \\
     & \implies  \| f_i + u_i\|^2 \leq \|f_i\|^2 \;\;\;\;\forall 1 \leq i \leq N \nonumber\\
     & \implies 2 Re\langle f_i , u_i \rangle + \|u_i\|^2 \leq 0 \;\;\;\;\forall 1 \leq i \leq N \nonumber\\
     & \implies 2\sum_{i=1}^N Re\langle f_i , u_i \rangle + \sum_{i=1}^N \|u_i\|^2 \leq 0
 \end{align}
  As $\displaystyle{\sum_{i=1}^N \langle f , f_i \rangle u_i =0}$ \;\; for all $f \in H, $ this implies $\displaystyle{\sum_{i=1}^N \langle f_i , u_i \rangle =0}.$\\~\\
  Therefore from equation (\ref{eqn4point11}) we have $\sum_{i=1}^N \|u_i\|^2 \leq 0.$
  This implies $u_i = 0,$ for all $1 \leq i \leq N.$\\
  Hence canonical dual is the only 1-erasure probabilistic optimal dual of $F.$
\end{proof}   \hfill{$\square$}

\begin{cor}
Let $H$ be an $n$ dimensional Hilbert space and $N \geq n.$ Let $\{q_i \}_{i=1}^N$ be a weight number sequence  given by  (\ref{eqn2point1}). . If a dual pair $(F, G)$ is an 1-erasure probabilistic optimal dual pair then it is probabilistic 1-uniform dual pair.
\end{cor}

\begin{proof}
 $(F,G) \in \Gamma_{1}^p \;$ implies \;\;$\|f_i \|\;\|g_i \| = \frac{1}{q_i}$ ;  $1 \leq i \leq N.$\\
 Therefore $$ n = \displaystyle{\sum_{i=1}^N \langle f_i , g_i \rangle} \leq   \displaystyle{\sum_{i=1}^N \left| \langle f_i , g_i \rangle \right|} \leq \displaystyle{\sum_{i=1}^N \|f_i \|\; \|g_i \| } = \displaystyle{\sum_{i=1}^N \frac{1}{q_i} = n }$$
 \noindent
 This implies $\langle f_i , g_i \rangle  =   \|f_i \| \;\|g_i \| = \frac{1}{q_i} \;, \;\; 1 \leq i\leq N.$\\~\\
 Let $\langle f_j , g_j \rangle = a_j + ib_j ,\;\;  1 \leq i\leq N.$ \\
 Then  $\displaystyle{\sum_{i=1}^N a_j =n }$,  $\displaystyle{\sum_{i=1}^N b_j =0 }$ \;\;and $\sqrt{a_j^2 + b_j^2} = \frac{1}{q_i}; \;\;  1 \leq i\leq N. $\\
 \noindent
 This gives $$ n= \displaystyle{\sum_{i=1}^N a_j } \leq  \displaystyle{\sum_{j=1}^N \sqrt{a_j^2 + b_j^2} } =  \displaystyle{\sum_{j=1}^N \frac{1}{q_i}} = n$$
 \noindent
 It is possible only when $a_j = \sqrt{a_j^2 + b_j^2} =  \frac{1}{q_i} $\;\;$1 \leq j \leq N,$ and  hence $b_j =0$\;\;\;$1 \leq j \leq N.$\\
 
 Therefore $ \langle f_J , g_J \rangle = \frac{1}{q_i},$ \;\;\;$1 \leq j \leq N.$
 
\end{proof}   \hfill{$\square$}

\noindent
The following theorem gives an equivalent condition between probabilistic optimal dual and probabilistic spectrally optimal dual for 1-erasure for a given parseval frame $F.$ This may be not true for any general frame.

\begin{thm} \label{thm4point2}
 Let $F = \{f_i\}_{i=1}^N $ be a  perseval frame for $H$ and  let $\{q_i\}_{i=1}^N$ be a weight number sequence given by  (\ref{eqn2point1}). TFAE :
 \begin{enumerate}
    \item [{\em (i)}] The canonical dual is an 1-erasure probabilistic optimal dual of $F.$
    \item [{\em (ii)}] The canonical dual is an 1-erasure probabilistic spectrally optimal dual of $F.$
\end{enumerate}
\end{thm}

\begin{proof}

Suppose the canonical dual is an 1-erasure probabilistic optimal dual of $F.$ Then for any dual $G = \{g_i\}_{i=1}^N$ of $F $ satisfy :
\noindent
\begin{equation} \label{eqn4point12}
  \displaystyle{max_{i=1}^N \;q_i \|f_i \|^2 \;\; \leq \;\; max_{i=1}^N \; q_i \|f_i \|\;\|g_i \|}  
\end{equation}

\noindent 
Let  $G' = \{g'_i\}_{i=1}^N =  \{f_i + u_i\}_{i=1}^N$ \; be a dual frame of $F$ such that 
\begin{equation}\label{eqn4point13}
    \displaystyle{ max_{i=1}^N \; q_i |\langle f_i , g'_i \rangle |\;\;  < \;\; max_{i=1}^N \;q_i \|f_i \|^2 }
\end{equation}

\noindent
Let $ c =  max_{i=1}^N q_i \|f_i \|^2 . $\\~\\
Consider $\Lambda_1 = \{i: q_i \|f_i \|^2 =  c\}$ and  $\Lambda_2 = \{ 1,2,...,N \} \setminus \Lambda_1 . $\\

\noindent
From (\ref{eqn4point13}) we have for all $i \in \Lambda_1$
$$ q_i |\langle f_i , g'_i \rangle | = q_i |\langle f_i , f_i +u_i \rangle | = \left| q_i \|f_i \|^2 + q_i \langle f_i , u_i \rangle \right| = \left| c + q_i \langle f_i , u_i \rangle \right| < c$$

\noindent
This implies $\mathrm{Re}\left( q_i \langle f_i , u_i \rangle \right)< 0 $ , and hence  $\mathrm{Re}\left(  \langle f_i , u_i \rangle \right)< 0 $ , for all $i \in \Lambda_1.$ \\

\noindent
As  $\mathrm{Re}(  \langle f_i , u_i \rangle )< 0 $ , for all $i \in \Lambda_1.$ we can take $\varepsilon _1 > 0 $ \; small enough such that for all $ i \in \Lambda_1$ \\
$$q_{i}^2 \| f_i + \varepsilon_1 u_i\|^2 \|f_i \|^2 = q_{i}^2 \left( \|f_i \|^2  + \varepsilon_1^2 \| u_{i}\|^2 + 2\varepsilon_1 \textit{Re}  \langle f_i , u_i \rangle \right) \|f_i \|^2 < q_{i}^2 \|f_i \|^4 = c^2$$ 

\noindent
Therefore  for all $ i \in \Lambda_1$
\begin{equation} \label{eqn4point14}
    q_{i} \| f_i + \varepsilon_1 u_i \|\;\|f_i \| < c
\end{equation}

\noindent
As\; $q_i \|f_i \|^2 < max_{j=1}^N q_j \|f_j \|^2 $ \; for all  $ i \in \Lambda_2,$ then we can choose $ \varepsilon_2 >0  $ small enough such that 
$$q_{i}^2 \| f_i + \varepsilon_2 u_i\|^2 \|f_i \|^2 = \left( q_{i}\|f_i \|^2  + q_{i} \varepsilon_2^2  \|u_{i}\|^2 + 2 q_{i} \varepsilon_2 \textit{Re}  \langle f_i , u_i \rangle \right) q_{i} \|f_i \|^2 < max_{j=1}^N\; q_{j}^2 \|f_i \|^4 = c^2.$$ 

\noindent
Hence  for all $ i \in \Lambda_2$
\begin{equation} \label{eqn4point15}
    q_{i} \| f_i + \varepsilon_1 u_i \|\; \|f_i \| < c
\end{equation}

\noindent
Take $ \varepsilon := min \{  \varepsilon_1 ,  \varepsilon_2 \}$ 
\noindent
Therefore the dual frame $G'' = \{ f_i + \varepsilon u_i \}_{i=1}^N $ of $F$ satisfies\\
$$q_i \|f_i + \varepsilon u_i \|\;\|f_i \| < c = max_{i=1}^N q_{i}\|f_i \|^2  \;\;\; \textit{for all } 1 \leq i \leq N.$$
This implies \;\; $max_{i=1}^N q_i \|f_i + \varepsilon u_i \|\;\|f_i \|  <   max_{i=1}^N q_{i}\|f_i \|^2 . $ Which is a contradiction.\\~\\
 Conversely, let the canonical dual is 1-erasure probabilistic spectrally optimal dual of $F.$\\
 Then for any dual $G = \{g_i \}_{i=1}^N $ \; of $H,$  
 $$  max_{i=1}^N q_{i}\|f_i \|^2 \leq  max_{i=1}^N q_i | \langle f_i , g_i \rangle | \leq  max_{i=1}^N q_i \|f_i \| \|g_i \| $$

\noindent 
 Hence the canonical dual is an 1-erasure probabilistic optimal dual of $F.$
 \end{proof} \hfill{$\square$}

 \begin{cor}\label{cor4point4}
  Let $F = \{f_i\}_{i=1}^N $ be a  perseval frame for $H$ and  let $\{q_i\}_{i=1}^N$ be a weight number sequence given by  (\ref{eqn2point1}). If the canonical dual of $F$ is the unique  1-erasure probabilistic spectrally optimal dual of $F,$ \; then it is the unique  1-erasure probabilistic optimal dual of $F$ and therefore for m-erasures.
 \end{cor}
 \begin{proof}
 As the canonical dual of $F$ is the unique  1-erasure probabilistic spectrally optimal dual, then for any other dual $G=\{ g_i \}_{i=1}^N$ of $F$\\
 $$ max_{i=1}^N q_{i}\|f_i \|^2 <  max_{i=1}^N q_i | \langle f_i , g_i \rangle | \leq  max_{i=1}^N q_i \|f_i \| \|g_i \| $$
 Hence the result follows.
 \end{proof} \hfill{$\square$}
 
 Note that converse of the corollary (\ref{cor4point4}) may not true always.
 \section{Examples}
 \begin{example}
  Let $ H= {\mathbb{C}}^2$ and consider a frame $F= \{f_1,f_2,f_3\},$ where $f_1 = \left[\begin{array}{l}
1 \\0
\end{array}\right] $,
$f_2 = \left[\begin{array}{l}
0 \\ 1
\end{array}\right] $
 and $ f_3 = \left[\begin{array}{l}
1 \\ 1
\end{array}\right] $ and the probability sequence $P=\{p_i\}_{i=1}^3$ is given by $p_1= p_2 = \frac{1}{4} , p_3 = \frac{1}{2}.$ Therefore the weight number sequence is  $\{q_i\}_{i=1}^3 = \{ \frac{4}{3} , \frac{4}{3}, 2\}.$\\

Then
$S_{F}^{-1} F = \bigg\{ \frac{1}{3}\left[\begin{array}{r} 2 \\ -1 \end{array}\right],  \frac{1}{3}\left[\begin{array}{r} -1 \\ 2 \end{array}\right], \frac{1}{3}\left[\begin{array}{l} 1 \\ 1  \end{array}\right] \bigg\}, $ \\
It is easy to calculate that \;\;$ q_1 \left| \langle S_{F}^{-1}f_1, f_1 \rangle \right|= q_2 \left| \langle S_{F}^{-1}f_2, f_2 \rangle \right|= \frac{8}{9} ,\;\; q_3 \left| \langle S_{F}^{-1}f_3, f_3 \rangle \right| = \frac{4}{3}. $\\
Therefore $ \mathcal{R}_{1}^p (F,S_{F}^{-1}F) = \frac{4}{3}.$\\

\noindent
As proposition (\ref{prop3point2}) $H_1 = span \left\{  \left[\begin{array}{l}
1 \\ 1
\end{array}\right]\right\}$ and $ H_2 = span\left\{ \left[\begin{array}{l}
1 \\0
\end{array}\right] , \left[\begin{array}{l}
0 \\ 1
\end{array}\right]\right\}.$  So $H_1 \cap H_2 \neq \{0\}.$\\
The set of duals of $F$ is of the form  $G= \{ g_i \}_{i=1}^3 = \bigg\{ \left[\begin{array}{r} \frac{2}{3} + \gamma  \\ -\frac{1}{3} + \delta \end{array}\right], \left[\begin{array}{r} -\frac{1}{3} + \gamma \\ \frac{2}{3} + \delta \end{array}\right], \left[\begin{array}{l} \frac{1}{3} + \gamma \\ \frac{1}{3} +\delta  \end{array}\right] \bigg\}; $ \; where $\gamma, \delta \in \mathbb{C}.$\\
\noindent
If we take $\gamma = \delta = \frac{1}{6},$\; then the dual is $G' = \{ g_i \}_{i=1}^3 = \bigg\{ \left[\begin{array}{r} \frac{5}{6}   \\ -\frac{1}{6}  \end{array}\right], \left[\begin{array}{r} -\frac{1}{6}  \\ \frac{5}{6}  \end{array}\right], \left[\begin{array}{l} -\frac{1}{6}  \\ -\frac{1}{6}  \end{array}\right] \bigg\}.$
It can be easily seen that $$ \mathcal{R}_{1}^p (F,G) = \frac{10}{9} < \frac{4}{3} = \mathcal{R}_{1}^p (F,S_{F}^{-1}F).$$ 
So the canonical dual is not the 1-erasure probabilistic spectrally optimal dual of $F.$\\~\\
Similarly as proposition (\ref{prop4point1})  $H_1 = span \left\{  \left[\begin{array}{l}
1 \\ 1
\end{array}\right]\right\}$ and $ H_2 = span\left\{ \left[\begin{array}{l}
1 \\0
\end{array}\right] , \left[\begin{array}{l}
0 \\ 1
\end{array}\right]\right\}.$  So $H_1 \cap H_2 \neq \{0\}.$  Now $ d_{1}^p (F, S_{F}^{-1}F) = \frac{4}{3} > \frac{2\sqrt{26}}{9} =  d_{1}^p (F, G).$\; Hence the canonical dual is not the 1-erasure probabilistic  optimal dual of $F.$

\end{example}

 \begin{example}
 Let $ H= {\mathbb{C}}^2$ and consider a frame $F= \{f_1,f_2,f_3,f_4\},$ where $f_1 = \left[\begin{array}{l}
1 \\0
\end{array}\right] $,
$f_2 = \left[\begin{array}{l}
0 \\ 1
\end{array}\right] $
 $ f_3 = \left[\begin{array}{l}
1 \\ 1
\end{array}\right] $and 
$f_4 = \left[\begin{array}{l}
1 \\ -1
\end{array}\right] $
and the probability sequence $P=\{p_i\}_{i=1}^4$ is given by $p_1=p_2=\frac{1}{2},\; p_3 = p_4 =0.$ Therefore the weight number sequence is  $\{q_i\}_{i=1}^4 = \{ 3,3,\frac{3}{2} , \frac{3}{2}\}.$\\
This is a tight frame with tight bound 3. Therefore $S_{F}^{-1} F = \left\{ \frac{1}{3}\left[\begin{array}{r} 1 \\ 0 \end{array}\right],  \frac{1}{3}\left[\begin{array}{r} 0 \\ 1 \end{array}\right], \frac{1}{3}\left[\begin{array}{l} 1 \\ 1  \end{array}\right],  \frac{1}{3}\left[\begin{array}{r} 1 \\ -1 \end{array}\right] \; \right\}, $ \\
This can be easily verified that $q_i | \langle f_i, S_{F}^{-1}f_i \rangle | =1,$ for all i. Therefore the canonical dual is an 1-erasure probabilistic spectrally optimal dual of $F.$\\
Morover the set of duals of $F$ is of the form  $G= \{ g_i \}_{i=1}^4 = \bigg\{ \left[\begin{array}{r} \frac{1}{3} + \alpha  \\ \beta \end{array}\right], \left[\begin{array}{r} -\alpha -2\gamma \\ \frac{1}{3} -\beta -2\delta  \end{array}\right], \left[\begin{array}{l} \frac{1}{3} + \gamma \\ \frac{1}{3} +\delta  \end{array}\right], \left[\begin{array}{l} \frac{1}{3} - \gamma -\alpha\\ -\frac{1}{3} -\delta -\beta \end{array}\right]  \bigg\}; $ \; where $\alpha, \beta, \gamma, \delta \in \mathbb{C}.$\\
For any value of  $\alpha, \beta, \gamma, \delta \in \mathbb{C}\setminus\{0\}$ \;\;$ \mathcal{R}_{1}^p (F,G) >  \mathcal{R}_{1}^p (F,S_{F}^{-1}F) .$ Therefore  canonical dual is the unique 1-erasure probabilistic spectrally optimal dual of $F.$\\
It is also easily varified that $q_i\|f_i\|\;\|S_{F}^{-1}f_i \| =1,$ for all i. Therefore the canonical dual is an unique 1-erasure probabilistic optimal dual of $F.$\\
Also by proposition (\ref{prop3point2}) $H_1 = span \left\{ f_1,f_2,f_3,f_4 \right\}$ and $H_2 = \{0\},$ so the canonical dual is an 1-erasure probabilistic spectrally optimal dual of $F.$ Similarly, by  proposition (\ref{prop4point1}) $H_1 = span \left\{ f_1,f_2,f_3,f_4 \right\}$ and $H_2 = \{0\},$ so the canonical dual is an 1-erasure probabilistic optimal dual of $F.$ Which satisfy the equivalent condition of  Theorem (\ref{thm4point2}).

 \end{example}

\section*{Acknowledgement}
The author is thankful to Prof. Devaraj P. for his guidance and help. The author is also thankful to Tathagata Sarkar and Abinash Sarma  for reading the manuscript and giving helpful inputs. The author is grateful to IISER-Thiruvananthapuram for providing fellowship.

\end{document}